\documentclass[12pt,a4paper]{amsart}
\usepackage[all,cmtip]{xy}
\usepackage{enumerate, comment}
\usepackage{mathrsfs}
\usepackage{ amssymb, latexsym, amsmath}
\usepackage{fullpage}
\usepackage{url}
\usepackage{hyperref}
\usepackage{helvet}
\usepackage{amsfonts}
\usepackage{tikz}
\usepackage{commath}
\usepackage{amsthm}
\usepackage{amscd}
\usepackage{stmaryrd}
\usepackage[mathscr]{eucal}
\usepackage{indentfirst}
\usepackage{graphicx}
\usepackage{tikz-cd}
\usepackage{graphics}
\usepackage{pict2e}
\usepackage{epic}
\newcommand{\closure}[2][3]{%
{}\mkern#1mu\overline{\mkern-#1mu#2}}
\usepackage[normalem]{ulem}
\usepackage[OT2,T1]{fontenc}
\numberwithin{equation}{section}
\usepackage[margin=3.2cm]{geometry}
 
\usepackage{color}
\DeclareSymbolFont{cyrletters}{OT2}{wncyr}{m}{n}
\DeclareMathSymbol{\Sha}{\mathalpha}{cyrletters}{"58}

\DeclareMathOperator{\Cl}{Cl}

\newcommand{\Z}{\mathbb{Z}}
\newcommand{\Q}{\mathbb{Q}}
\newcommand{\F}{\mathbb{F}}
\newcommand{\Kan}{K^{\operatorname{ac}}}
\newcommand{\op}[1]{\operatorname{#1}}
\newcommand\mtx[4] { \left( {\begin{array}{cc}
   #1 & #2 \\
   #3 & #4 \\
  \end{array} } \right)}
\newcommand{\DK}[1]{\textcolor{purple}{#1}}
\newcommand{\AR}[1]{\textcolor{red}{#1}}

\theoremstyle{plain}
 \theoremstyle{definition}
\newtheorem{Th}{Theorem}[section]
\newtheorem{Lemma}[Th]{Lemma}
\newtheorem{Example}[Th]{Example}
\newtheorem{hypothesis}[Th]{Hypothesis}
 
\newtheorem{Corollary}[Th]{Corollary}
\newtheorem{Proposition}[Th]{Proposition}
\newtheorem{Remark}[Th]{Remark}
 \theoremstyle{definition}
\newtheorem{Definition}[Th]{Definition}

\newtheorem{Fact}[Th]{Fact}
\begin{document}

\title{Anticyclotomic {\Large{$\mu$}}-invariants of residually reducible Galois Representations}
\author{Debanjana Kundu}
\address{Department of Mathematics \\ University of British Columbia \\
  Vancouver BC, V6T 1Z2, Canada.} \email{dkundu@math.ubc.ca}
\author{Anwesh Ray}
\address{Department of Mathematics \\ University of British Columbia \\
  Vancouver BC, V6T 1Z2, Canada.} 
  \email{anweshray@math.ubc.ca}
\begin{abstract} Let $E$ be an elliptic curve over an imaginary quadratic field $K$, and $p$ be an odd prime such that the residual representation $E[p]$ is reducible.
The $\mu$-invariant of the fine Selmer group of $E$ over the anticyclotomic $\Z_p$-extension of $K$ is studied.
We do not impose the Heegner hypothesis on $E$, thus allowing certain primes of bad reduction to decompose infinitely in the anticyclotomic $\Z_p$-extension.
It is shown that the fine $\mu$-invariant vanishes if certain explicit conditions are satisfied.
Further, a partial converse is proven.
\end{abstract}

\date{\today}

\maketitle
\section{Introduction}
Iwasawa theory is the study of objects of arithmetic interest over infinite towers of number fields.
K. Iwasawa conjectured that over the \textit{cyclotomic} $\Z_p$-extension of a number field $F$, the $p$-primary part of the Hilbert class group is a finitely generated $\Z_p$-module (see \cite{Iwa73}).
This is known as Iwasawa's $\mu=0$ conjecture for the number field $F$.
In \cite{FW79}, B. Ferrero and L. Washington proved this conjecture for all finite abelian extensions $F$ over of $\Q$.
The Iwasawa theory of abelian varieties (in particular, elliptic curves) was initiated by B. Mazur in \cite{Maz72}.
The main object of study is the $p$-primary Selmer group of an elliptic curve $E$ defined over a number field $F$, with good \textit{ordinary} reduction at $p$.
The Selmer group over a $\Z_p$-extension of $F$ is a cofinitely generated module over the Iwasawa algebra, $\Lambda$.
In \cite{Gre89}, R. Greenberg analyzed the algebraic structure of these Selmer groups.
%Denote by $\mu(E/F^{\op{cyc}})$ the $\mu$-invariant of the $p$-primary Selmer group of $E$ over the cyclotomic $\Z_p$-extension, $F^{\op{cyc}}$.
For elliptic curves $E$ over $\Q$, it is conjectured that if the residual representation on $E[p]$ is irreducible, then $\mu$-invariant of the $p$-primary Selmer group over the cyclotomic $\Z_p$-extension, denoted $\mu(E/\Q^{\op{cyc}})$, vanishes (see \cite[Conjecture 1.11]{Gre99}). 
\par The \textit{fine} Selmer group is a subgroup of the classical Selmer group obtained by imposing vanishing conditions at primes above $p$.
It had first been studied by K. Rubin \cite{Rub14} and B. Perrin-Riou \cite{PR93, PR95}, under various guises.
The analysis of the fine Selmer groups is an essential part of K. Kato's seminal work on the Iwasawa Main Conjecture for elliptic curves and modular forms (see \cite{Kat04}).
In recent years, their study has gained considerable momentum (see for instance \cite{CS05, Wut05, Wut07, JS11, Ari14}).
J. Coates and R. Sujatha conjectured that the the fine Selmer group of $E$ over $F^{\op{cyc}}$ is $\Lambda$-cotorsion with associated $\mu$-invariant, $\mu^{\op{fine}}(E/F^{\op{cyc}})$, equal to zero \cite[Conjecture A]{CS05}.
The formulation of this conjecture makes no hypothesis on the reduction type at primes above $p$ or the residual representation on $E[p]$.
However, such a statement need not be true for the (classical) Selmer group even for elliptic curves over $\Q$ with good ordinary reduction at $p$.
In fact, Mazur provided examples of elliptic curves $E_{/\Q}$ for which the residual representation $E[p]$ is reducible and $\mu(E/\Q^{\op{cyc}})$ is non-zero.
There is a systematic approach towards finding examples of elliptic curves with positive $\mu(E/F^{\op{cyc}})$ (see \cite{Dri02, Dri03}).

On the other hand, the conjecture by Coates and Sujatha predicts a close relationship in the growth of ideal class groups and fine Selmer groups in cyclotomic $\Z_p$-extensions.
Some evidence towards this is provided in \cite[Theorem 3.4]{CS05}.
In particular, the result of Ferrero and Washington implies that for an elliptic curve over an abelian number field, $\mu^{\op{fine}}(E/F^{\op{cyc}})$ is zero when $E[p]$ is reducible (see \cite[Corollary 3.6]{CS05}).
Subsequently, the relation in the growth of ideal class groups and fine Selmer groups has been studied in more general settings (see \cite{LM15, Kun20_infinite, Kun20_p, Kun20_uniform}).
In this paper, we study the relationship in the growth of fine Selmer groups and class groups in \textit{anticyclotomic} $\Z_p$-extensions of imaginary quadratic fields.

Several authors have studied classical Selmer groups of elliptic curves (more generally, abelian varieties or modular forms) in anticyclotomic $\Z_p$-extensions of imaginary quadratic fields (see for example \cite{Vat03, BD05, PW11}).
In line with the conjecture of Coates and Sujatha, for an elliptic curve defined over the imaginary quadratic field $K$, its fine Selmer group over the anticyclotomic $\Z_p$-extension is expected to be $\Lambda$-cotorsion with $\mu^{\op{fine}}(E/\Kan)=0$ (see \cite[Conjecture B]{Mat18}).
Most results in literature focus on the case when the residual representation is irreducible.
In contrast, this paper primarily studies the case when the residual representation is \textit{reducible}.
The first main result (see Theorem~\ref{theorem1}) shows that if certain conditions are satisfied then the $\mu$-invariant of the $p$-primary fine Selmer group may be detected by a certain analogous fine Selmer group associated to the residual representation, $E[p]$.
Further, if $E[p]$ is reducible then there are characters \[\varphi_1,\varphi_2:\op{Gal}(\closure{K}/K)\rightarrow \op{GL}_1(\F_p)\] which fit into a short exact sequence \[0\rightarrow \F_p(\varphi_1)\rightarrow E[p]\rightarrow \F_p(\varphi_2)\rightarrow 0.\] Note that $\varphi_2=\overline{\chi}\varphi_1^{-1}$, where $\overline{\chi}$ denotes the mod-$p$ cyclotomic character.
This short exact sequence makes it possible to analyze the structure of the fine Selmer group associated to $E[p]$.
Next, we prove a partial converse to the above theorem (see Theorem \ref{theorem2}).
A key difference between the cyclotomic and anticyclotomic $\Z_p$-extensions is the following: in the cyclotomic extension, all primes are finitely decomposed, whereas in the anticyclotomic extension, there are infinitely many primes which split completely.
For the aforementioned result of Coates and Sujatha in the cyclotomic extension case, this fact is crucially used (in the proofs of \cite[Lemmas 3.2, 3.7, Theorem 3.4]{CS05}).
The elliptic curve $E$ is said to satisfy the Heegner hypothesis if all primes $v\nmid p$ of $K$ at which it has bad reduction are split completely in $K$.
By Fact $\ref{fact1}$, such a a prime number is finitely decomposed in the anticyclotomic $\Z_p$-extension, $\Kan$.
%When the Heegner hypothesis is satisfied, the $\mu$-invariant $\mu^{\op{fine}}(E/\Kan)$ may be analyzed via the aforementioned arguments of Coates and Sujatha (see also \cite[section1]{CGLS20}).
%When primes of bad reduction of $E$ are finitely decomposed in $\Kan$, the results may be obtained via arguments that are to some extent well understood; however, the results for the fine Selmer group have not been documented.
We work in a more general setting, where primes of bad reduction of $E$ may be infinitely decomposed in $\Kan$, and thereby rely on a different strategy (from that of Coates and Sujatha) to prove our results.

\par Two elliptic curves, $E_1$ and $E_2$, are said to be $p$-congruent if the $p$-torsion subgroups $E_1[p]$ and $E_2[p]$ are isomorphic as Galois modules.
In \cite{GV00}, R. Greenberg and V. Vatsal studied the relation between cyclotomic invariants of elliptic curves over $\Q$ which are $p$-congruent.
In \cite[Proposition 4.1.6]{Gre11}, Greenberg reformulated the conjecture of Coates and Sujatha in terms of the vanishing of a second cohomology group with coefficients in $E[p]$.
From this, it is immediate that if elliptic curves $E_1$, $E_2$ are $p$-congruent then $\mu^{\op{fine}}(E_1/\Q^{\op{cyc}})=0$ if and only if $\mu^{\op{fine}}(E_2/\Q^{\op{cyc}})=0$.
In the same spirit, we prove a result for fine $\mu$-invariants over $\Kan$ (see Theorem \ref{theorem3}), without imposing any hypothesis on the reducibility of the residual representations.
\par The authors expect that the methods in this paper should generalize to residually reducible Galois representations arising from abelian varieties.
The results shall however be more technical and the arguments more cumbersome in the higher dimensional setting.
The authors choose a less general framework in which the inherent simplicity of the underlying ideas come across easily.

\par The paper is organized into six sections.
Preliminary notions are discussed in \S\ref{Preliminaries}.
In \S\ref{A Criterion for Vanishing} and \S\ref{finiteness of the residual fine Selmer group}, we state and prove the main results: we establish a criterion for the vanishing of the $\mu$-invariant of the fine Selmer group in the anticyclotomic $\Z_p$-extension.
We also prove a partial converse to the above theorem.
In \S\ref{Congruent Galois representations}, we compare the anticyclotomic $\mu$-invariant for two elliptic curves which are $p$-congruent.
Finally, in \S\ref{examples} we list examples illustrating the results in this article.

\section{Preliminaries}
\label{Preliminaries}
\par Throughout, let $p$ be an odd prime and $K$ be an imaginary quadratic field.
Let $S_p$ be the set of primes of $K$ above $p$.
Fix an elliptic curve $E$ over $K$.
Set $\op{G}_{K}:= \op{Gal}(\closure{K}/K)$ to be the absolute Galois group of $K$.
Denote by $E[p^n]$ (resp. $E[p^{\infty}]$) the $p^n$-torsion (resp. $p$-primary torsion) subgroup of $E(\closure{K})$.
The Tate module $\op{T}_p(E)$ is the inverse limit with respect to multiplication by $p$ maps, \[\op{T}_p(E):=\varprojlim_n E[p^n].\]
The $\Z_p$-module $\op{T}_p(E)$ is free of rank $2$ and the group of $\Z_p$-linear automorphisms of $\op{T}_p(E)$ is identified with $\op{GL}_2(\Z_p)$.
The action of $\op{G}_{K}$ on $\op{T}_p(E)$ induces a continuous Galois representation, $\rho_E:\op{G}_{K}\rightarrow \op{GL}_2(\Z_p)$.
Set $\closure{\rho}_E$ to denote the mod-$p$ reduction of $\rho_E$, as depicted
 \[ \begin{tikzpicture}[node distance = 2.5 cm, auto]
            \node at (0,0) (G) {$\op{G}_{K}$};
             \node (A) at (3,0){$\op{GL}_2(\F_p)$.};
             \node (B) at (3,2){$\op{GL}_2(\Z_p)$};
      \draw[->] (G) to node [swap]{$\closure{\rho}_E$} (A);
       \draw[->] (B) to node{} (A);
      \draw[->] (G) to node {$\rho_E$} (B);
\end{tikzpicture}\]
The residual representation $\closure{\rho}_E$ is induced by the action of $\op{G}_{K}$ on $E[p]$.
Let $\mathcal{N}$ be the conductor of $E$ and $S$ the set of primes that divide $\mathcal{N}p$.
It is known that $\rho_E$ is unramified at all primes $v\notin S$.
Let $K_S$ denote the maximal algebraic extension of $K$ in which all primes $v\notin S$ are unramified and set $\op{G}_{K,S}:=\op{Gal}(K_S/K)$.
Throughout this paper, we primarily focus on the case when the Galois representation $\rho_E$ is residually reducible.
However, some of the results will apply to the residually irreducible case as well.
We shall be careful to make precise when the following hypothesis is required.
 \begin{hypothesis}
Assume that the residual representation $\closure{\rho}_E$ is reducible, i.e. $E[p]$ contains a proper non-zero $\op{G}_{K}$-stable submodule.
\end{hypothesis}
This is indeed the case when $E(K)[p]\neq 0$. 
Though the condition is in fact far more general.
In this setting, there are characters \[\varphi_1, \ \varphi_2:\op{G}_{K,S}\rightarrow \op{GL}_1(\F_p)\] and a $1$-cocycle \[\beta:\op{G}_{K,S}\rightarrow \F_p(\varphi_1\varphi_2^{-1})\] such that
\[\closure{\rho}_E\simeq \mtx{\varphi_1}{\varphi_2 \beta}{}{\varphi_2}.\]
The residual representation $\closure{\rho}_E$ is said to be \textit{indecomposable} if the cohomology class $[\beta]\in H^1(\op{G}_{K,S}, \F_p(\varphi_1\varphi_2^{-1}))$ is non-zero, and \textit{split} otherwise.
If $\closure{\rho}_E$ is split, then $\closure{\rho}_E\simeq \mtx{\varphi_1}{}{}{\varphi_2}$, and the characters $\varphi_1$ and $\varphi_2$ may be interchanged.
If $\closure{\rho}_E$ is indecomposable, then $\varphi_1$ is the unique character such that $E[p]$ contains a Galois submodule isomorphic to $\F_p(\varphi_1)$.
\begin{Remark}If $E_{/\Q}$ is an elliptic curve such that $E[p]$ is reducible as a $\op{Gal}(\closure{\Q}/\Q)$-module then $E$ has either good ordinary reduction or bad reduction at $p$.
By a result of Fontaine, if $E_{/\Q}$ is an elliptic curve with good supersingular reduction at $p$, then $\closure{\rho}_{E\restriction \op{G}_{\Q_p}}$ is irreducible.\end{Remark}

\par Suppose $\mathcal{K}$ is a $\Z_p$-extension of $K$ which is Galois over $\Q$.
Then, there are exactly two cases to consider: either $\mathcal{K}$ is the cyclotomic $\Z_p$-extension or the anticyclotomic $\Z_p$-extension.
Note that $\Gamma_{\mathcal{K}}:=\op{Gal}(\mathcal{K}/K)$ is an index two normal subgroup of $\op{Gal}(\mathcal{K}/\Q)$.
The group $\op{Gal}(K/\Q)$ acts on $\Gamma_{\mathcal{K}}$, as is explained.
Let $\tau\in \op{Gal}(K/\Q)$ and $x\in \Gamma_{\mathcal{K}}$; choose a lift $\tilde{\tau}\in \op{Gal}(\mathcal{K}/\Q)$ of $\tau$ and set $\tau \cdot x:=\tilde{\tau} x\tilde{\tau}^{-1}$.
Since $\Gamma_{\mathcal{K}}$ is abelian, $\tau\cdot x$ does not depend on the choice of the lift $\tilde{\tau}$.
If the action of $\op{Gal}(K/\Q)$ on $\Gamma_{\mathcal{K}}$ is via the trivial (resp. non-trivial) character, then $\mathcal{K}$ is the cyclotomic (resp. anticyclotomic) extension of $K$.
Thus, the cyclotomic extension is pro-cyclic and the anticyclotomic extension is pro-dihedral.
We only consider the anticyclotomic extension in this article, and denote it by $\Kan$.
Set $\Gamma$ to denote $\op{Gal}(\Kan/K)$.
For $n\geq 0$, the \textit{$n$-th layer} is the unique number field $K_n$ such that $K\subseteq K_n \subset \Kan$ and $[K_n:K]=p^n$.
Note that $K_n$ is Galois over $\Q$ and its Galois group $\op{Gal}(K_n/\Q)$ is (isomorphic to) the dihedral group of order $2p^n$.
\par For a set of primes $S'$, set $S'(\Kan)$ (resp. $S'(K_n)$) to denote the primes of $\Kan$ (resp. $K_n$) which lie above some prime $v\in S'$.
For instance, $v(\Kan)$ (resp. $v(K_n)$) will denote the primes $\eta|v$ of $\Kan$ (resp. of $K_n$).
The set of primes above a given prime of $K$ in the cyclotomic $\Z_p$-extension $K^{\op{cyc}}$ is finite.
This is not the case for $\Kan$; the following characterization is well known. 
\begin{Fact}\label{fact1}Let $v$ be a prime of $K$ and $l$ the prime number such that $v|l$.
The set of primes $v(\Kan)$ is finite if and only if $l=p$ or $l$ splits in $K$.
%\cite[p.388]{Oza01} and \cite[pp.5-6]{Iwa73_CounterAC}.
\end{Fact}
%\AR{I wrote this out in detail because Corollory 1 of the Brink paper is confusing since it seems to suggest that all primes $p$ are finitely split in $\Kan$. The other assertions are not easy to locate.}
Indeed, there is a large enough value of $n$ such that all primes of $K_n$ above $p$ are totally ramified in $\Kan$, see the last three lines of \cite[p. 2131]{Bri07}.
A prime $l\neq p$ which does not split in $K$ must split completely in $\Kan$, see the first paragraph of p. 2132 of \textit{loc. cit.}
Prime numbers $l$ which split in $K$ must be finitely decomposed in $\Kan$, see for example, Corollary 1 of \textit{loc. cit.}

Recall that $\mathcal{N}$ is the conductor of $E$.
Denote by $\closure{\mathcal{N}}$ the Artin conductor of $\closure{\rho}_{E}$.
Note that $\closure{\mathcal{N}}$ divides $\mathcal{N}$, denote by $(\mathcal{N}/\closure{\mathcal{N}})$ the quotient.
We make an assumption on a certain subset of primes in $S$ which is described below.
\begin{Definition}
Denote by $\Sigma\subset S$ the set of primes $v\nmid  p$ at which \textit{all} of the following conditions are satisfied.
\begin{enumerate}
\item $v|(\mathcal{N}/\closure{\mathcal{N}})$.
\item If $p\geq 5$ and $\mu_p \subset K_v$, then $E$ has split multiplicative reduction at $v$.
\item If $p=3$ and $\mu_3\subset K_v$, then $E$ has split multiplicative reduction or additive reduction at $v$.
\end{enumerate}

\end{Definition}
\begin{hypothesis}\label{splithyp}
Let $E$ be an elliptic curve over an imaginary quadratic field $K$.
Let $v\in \Sigma$ and $l$ be the prime number such that $v|l$.
Then $l$ is split in $K$.
\end{hypothesis}

\begin{Definition}
Suppose that $\closure{\rho}_E$ is reducible and indecomposable.
Let $\Sigma(\varphi_2)$ be the set of primes $v\in S\setminus S_p$ such that $\varphi_{2\restriction \op{G}_{\Q_v}}=1$.
\end{Definition}

\begin{hypothesis}\label{splithyp2}
Suppose that $\closure{\rho}_E$ is reducible and indecomposable.
%This is a hypothesis on $\varphi_2$.
Let $v\in \Sigma(\varphi_2)$ and $l$ be the prime number in such that $v|l$.
Then $l$ is split in $K$. 
\end{hypothesis}

\begin{Remark}
The above hypotheses are all weaker than the Heegner hypothesis.
Some of our results will require the above hypotheses; they will be assumed only when explicitly stated.
For instance, Theorem $\ref{theorem1}$ will require Hypothesis $\ref{splithyp}$ (and Hypothesis $\ref{splithyp2}$ when $\closure{\rho}_E$ is indecomposable).
On the other hand, Theorem $\ref{theorem2}$ does not require these hypotheses.
\end{Remark}
We now introduce the fine Selmer group.
At each prime $v\in S$, set \[\mathcal{H}_v(\Kan, E[p^{\infty}]) :=\prod_{\eta\in v(\Kan)} H^1(\Kan_{\eta}, E[p^{\infty}]).\]

\begin{Definition}The \textit{fine Selmer group} associated to $E[p^{\infty}]$ is defined as follows
\[\mathcal{R}_{p^{\infty}}(E/\Kan) := \op{ker}\left\{ H^1(K_S/\Kan, E[p^{\infty}])\longrightarrow \bigoplus_{v\in S} \mathcal{H}_v( \Kan, E[p^{\infty}]) \right\}.\]
\end{Definition}

Recall that $\Gamma:=\op{Gal}(\Kan/K)\simeq \Z_p$.
The Iwasawa algebra $\Lambda$ is the completed group algebra $\Z_p\llbracket \Gamma \rrbracket :=\varprojlim_n \Z_p[\Gamma/\Gamma^{p^n}]$.
After fixing a topological generator $\gamma$ of $\Gamma$, there is an isomorphism of rings $\Lambda\cong\Z_p\llbracket X\rrbracket $, by sending $\gamma -1$ to the formal variable $X$.
The fine Selmer group is a cofinitely generated $\Lambda$-module.

The Pontryagin dual $\mathcal{R}_{p^{\infty}}(E/\Kan)^{\vee}$ is, up to pseudo-isomorphism, a finite direct sum of cyclic $\Lambda$-modules:
\[
\mathcal{R}_{p^{\infty}}(E/\Kan)^{\vee}\sim \Lambda^r\oplus \left(\bigoplus_{i=1}^s \Lambda/(p^{\mu_i})\right)\oplus \left(\bigoplus_{j=1}^t \Lambda/(f_j(T)) \right).
\]
Here, $\mu_i>0$ and $f_j(T)$ is a distinguished polynomial (i.e. a monic polynomial with non-leading coefficients divisible by $p$).
Here, $r$ is the $\Lambda$-corank of the fine Selmer group.
The $\mu$-invariant of the fine Selmer group is defined as follows,
\[
\mu^{\op{fine}}(E/\Kan):=\begin{cases}0 & \textrm{ if } s=0\\
\sum_{i=1}^s \mu_i & \textrm{ if } s>0.
\end{cases}
\]
The number of summands $s$ is a well-defined invariant and we refer to it as the \textit{$\mu$-multiplicity}. %; it is denoted $s^{\op{fine}}(E/\Kan)$.
We now introduce the fine Selmer group associated to the residual representation on $E[p]$.
At each prime $v$ of $K$, set \[\mathcal{H}_v(\Kan, E[p]):=\prod_{\eta\in v(\Kan)} H^1(\Kan_{\eta}, E[p]).\]

\begin{Definition}\label{Tfine}Let $T$ be a finite set of primes containing $S$. The fine Selmer group associated to $E[p]$ and the set of primes $T$ is defined as follows
\[\mathcal{R}^T(E[p]/\Kan) := \op{ker}\left\{ H^1(K_T/\Kan, E[p])\longrightarrow \bigoplus_{v\in T} \mathcal{H}_v( \Kan, E[p]) \right\}.\]
Set $\mathcal{R}(E[p]/\Kan) = \mathcal{R}^S(E[p]/\Kan)$; this is the mod-$p$ fine Selmer group.
\end{Definition}

Set $\Omega$ to denote the mod-$p$ Iwasawa algebra, $\Omega:=\Lambda/(p)\simeq \F_p\llbracket X \rrbracket$.
Note that both $\mathcal{R}(E[p]/\Kan)$ and $\mathcal{R}_{p^{\infty}}(E/\Kan)[p]$ are $\Omega$-modules.
Since $\mathcal{R}_{p^{\infty}}(E/\Kan)$ is cofinitely generated over $\Lambda$, it follows that $\mathcal{R}_{p^{\infty}}(E/\Kan)[p]$ is cofinitely generated over $\Omega$.
The following is an easy consequence of the structure theory of $\Lambda$-modules.

\begin{Lemma}\label{lemma27}
The $\Omega$-corank of $\mathcal{R}_{p^{\infty}}(E/\Kan)[p]$ is equal to $r+s$, where $r$ and $s$ are defined above.
In particular,
\[\mathcal{R}_{p^{\infty}}(E/\Kan) \text{ is }\Lambda\text{-cotorsion with }\mu^{\op{fine}}(E/\Kan)=0\Leftrightarrow \mathcal{R}_{p^{\infty}}(E/\Kan)[p]\text{ is finite.}\]
\end{Lemma}
\begin{proof}
Note that $\left(\mathcal{R}_{p^{\infty}}(E/\Kan)[p]\right)^{\vee}$ is isomorphic to $\mathcal{R}_{p^{\infty}}(E/\Kan)^{\vee}/p$.
Let $\Phi$ be a pseudo-isomorphism of $\Lambda$-modules, i.e. a homomorphism, \[\Phi:\mathcal{R}_{p^{\infty}}(E/\Kan)^{\vee}\rightarrow \Lambda^r\oplus \left(\bigoplus_{i=1}^s \Lambda/(p^{\mu_i}) \right)\oplus \left(\bigoplus_{j=1}^t \Lambda/(f_j(T)) \right)\] with finite kernel and cokernel.
The mod-$p$ reduction $\overline{\Phi}$ is the following map
\[\closure{\Phi}:\left(\mathcal{R}_{p^{\infty}}(E/\Kan)[p]\right)^{\vee}\rightarrow \Omega^{r+s}\oplus \left(\bigoplus_{j=1}^t \Omega/(\closure{f}_j(T)) \right),\]where $\closure{f}_j(T)$ is the mod-$p$ reduction of $f_j(T)$.
Since $f_j(T)$ is a distinguished polynomial, $\closure{f}_j(T)=T^{\op{deg}f_j}$ and $\Omega/(\closure{f}_j(T))$ is finite.
The result follows.
\end{proof}

\begin{comment}
\begin{Lemma}
The following assertions hold:
\begin{enumerate}
    %\item Suppose that $l\neq p$ and $l$ is inert or ramified in $K$.
    Then either $\mathcal{H}_v(\Kan, E[p^{\infty}])$ is equal to $0$ or is cotorsion over $\Lambda$.
    \item Suppose that $l=p$ or $l$ splits in $K$.
    Then, $\mathcal{H}_v( \Kan, E[p^{\infty}])$ is a cotorsion $\Lambda$-module with $\mu$-invariant equal to zero.
\end{enumerate}
\end{Lemma}
\begin{proof}
\par Suppose that $l\neq p$ and $l$ is inert or ramified in $K$.
Let $\eta\in v(\Kan)$ and $\Gamma_v\subset \Gamma$ be the decomposition subgroup of $\eta|v$.
Since $v$ is infinitely decomposed in $\Kan$, the group $\Gamma_v$ is finite.
There is a natural isomorphism:
\[\mathcal{H}_v( \Kan, E[p^{\infty}])\xrightarrow{\sim} \Lambda\otimes_{\Lambda(\Gamma_{v})} H^1(\Kan_{\eta}, E[p^{\infty}]).\] Therefore, either $\mathcal{H}_v(\Kan, E[p^{\infty}])$ is equal to $0$ or is cotorsion over $\Lambda$. 
\par Suppose that $l=p$ or $l$ splits in $K$.
By the previous lemma, it suffices to show that $\mathcal{H}_v( \Kan, E[p^{\infty}])[p]$ is finite.
The sequence \[0\rightarrow E[p]\rightarrow E[p^{\infty}]\xrightarrow{\times p} E[p^{\infty}]\rightarrow 0\] induces a surjective map 
\[\mathcal{H}_v( \Kan, E[p])\rightarrow \mathcal{H}_v( \Kan, E[p^{\infty}])[p].\]
Therefore, it is enough to show that $\mathcal{H}_v( \Kan, E[p])$ is finite.
By Fact $\ref{fact1}$, there are only finitely many primes in $v(\Kan)$ and therefore, the result follows.
\end{proof}
\end{comment}
\section{A Criterion for the Vanishing of the fine {\large $\mu$}-invariant}
\label{A Criterion for Vanishing}
In this section, we establish a criterion for the vanishing of the fine $\mu$-invariant in the anticyclotomic $\Z_p$-extension.
Before stating the result, let us introduce some notation.
We shall not assume that $\closure{\rho}_E$ is reducible in this section, unless when explicitly stated.
However, when it is reducible, recall that the characters $\varphi_1$ and $\varphi_2$ are such that $E[p]$ fits into a short exact sequence of Galois-modules,
\[0\rightarrow \F_p(\varphi_1)\rightarrow E[p]\rightarrow \F_p(\varphi_2)\rightarrow 0.\]
Let $K(\varphi_i)$ be an extension of $K$ fixed by $\ker \varphi_i$.
For an algebraic extension $\mathcal{F}$ of $K$, let $\mathcal{F}(\varphi_i)$ be the composite $\mathcal{F}\cdot K(\varphi_i)$.
Let $\mathcal{L}_n^{(i)}$ be the maximal abelian unramified $p$-extension of $K_n(\varphi_i)$ split at primes above $S(K_n(\varphi_i))$ (called the $p$-Hilbert $S$-class field extension).
Class field theory prescribes a natural isomorphism 
\[\op{Gal}(\mathcal{L}_n^{(i)}/K_n(\varphi_i))\simeq \op{Cl}_S(K_n(\varphi_i))[p^{\infty}],\]where $\op{Cl}_S(K_n(\varphi_i))[p^{\infty}]$ is the $p$-primary part of the $S$-class group of $K_n(\varphi_i)$.

Let $\rm{M}$ be any $\Z_p$-module on which $\op{Gal}(K(\varphi_i)/K)$ acts by $\Z_p$-linear automorphisms.
Let $\psi:\op{Gal}(K(\varphi_i)/K)\rightarrow \F_p^{\times}$ be a character, and consider its Teichm\"uller lift, $\tilde{\psi}:\op{Gal}(K(\varphi_i)/K)\rightarrow \Z_p^{\times}$.
Set
\[\textrm{M}_{\psi}:=\{ x\in \textrm{M}| g\cdot x=\tilde{\psi}(g)x\}.\]
%Denote by $\left(\op{Cl}_S(K_n(\varphi_i))[p^{\infty}]\right)_{\varphi_i}$ the $\varphi_i$-component of $\op{Cl}_S(K_n(\varphi_i))[p^{\infty}]$, i.e.
%\[\begin{split}\left(\op{Cl}_S(K_n(\varphi_i))[p^{\infty}]\right)_{\varphi_i} :=\{x\in \op{Cl}_S(K_n(\varphi_i))[p^{\infty}] & \mid g\cdot x = \varphi_i(g) x \\ & \text{ for }g\in \op{Gal}(K_n(\varphi_i)/ K)\}.\end{split}\]
There is a unique field extension $\mathcal{E}_n^{(i)}$ of $K_n(\varphi_i)$ which is Galois over $K$ with the additional property that $\op{Gal}(\mathcal{E}_n^{(i)}/K_n(\varphi_i))$ is identified with $\left(\op{Cl}_S(K_n(\varphi_i))[p^{\infty}]\right)_{\varphi_i}$.
Set $\mathcal{L}^{(i)}:=\bigcup_{n} \mathcal{L}^{(i)}_n$ and $\mathcal{E}^{(i)}:=\bigcup_{n} \mathcal{E}^{(i)}_n$.
Further, setting $\mathcal{X}^{(i)}=\op{Gal}(\mathcal{L}^{(i)}/\Kan(\varphi_i))$, 
standard arguments show that $\mathcal{X}^{(i)}$ is a finitely generated torsion $\Lambda$-module. 
Note that $\mathcal{X}^{(i)}$ decomposes into a direct sum of $\Lambda$-submodules
\[\mathcal{X}^{(i)}=\bigoplus_{\psi} \mathcal{X}^{(i)}_{\psi},\]where $\psi$ ranges over the characters $\op{Gal}(K(\varphi_i)/K)\rightarrow \F_p^{\times}$.
Recall that the $\psi$-eigenspace $\mathcal{X}^{(i)}_{\psi}$ is defined as follows
\[\mathcal{X}^{(i)}_{\psi}:=\{x\in \mathcal{X}^{(i)}\mid g \cdot x= \tilde{\psi}(g) x\}.\]
Let $Y_i$ be the $\varphi_i$-eigenspace $\mathcal{X}^{(i)}_{\varphi_i}$ and $\widetilde{Y}_i$ be the $\varphi_i$-component of the $p$-Hilbert class field extension $\mathcal{M}^{(i)}$ of $\Kan(\varphi_i)$.
Denote their $\mu$-invariants (with respect to the $\Z_p$-extension $\Kan(\varphi)/K(\varphi_i)$) by $\mu(Y_i)$ and $\mu(\widetilde{Y_i})$, respectively.
Note that $Y_i$ is the quotient of $\widetilde{Y}_i$ where all the primes in $S(\Kan(\varphi_i))$ are split.
We identify $Y_i$ with the Galois group $\op{Gal}(\mathcal{E}^{(i)}/\Kan)$.
The next result follows from the work of H. Hida \cite[Theorem I]{Hid10} (see also \cite[Theorem 1.1]{Fin06}) and Rubin \cite[Theorem 4.1]{Rub91}.
\begin{Th}\label{hidarubin}
Suppose that $p = v \overline{v}$ in $K$.
Let $\op{G}_{v}=\op{Gal}(\closure{K_{v}}/K_v)$ and assume that $\varphi_{1\restriction \op{G}_{v}}\neq 1, \ \overline{\chi}_{\restriction \op{G}_{v}}$ (or equivalently, $\varphi_{2\restriction \op{G}_{v}}\neq 1, \ \overline{\chi}_{\restriction \op{G}_{v}}$).
Then, 
\[\mu(\widetilde{Y}_1)=\mu(\widetilde{Y}_2)=\mu(Y_1)=\mu(Y_2)=0.\]
\end{Th}
%\DK{it might be obvious but I need to think why the assumption on $\varphi_1$ and $\varphi_2$ are equivalent.}
%\AR{cuz the determinant is $\overline{\chi}$, so $\varphi_2=\overline{\chi} \varphi_1^{-1}$. This is noted earlier.}
\begin{proof}
Since $\varphi_1\varphi_2=\overline{\chi}$, the condition on $\varphi_1$ is equivalent to that on $\varphi_2$.
Let $\varphi_i$ be either of the characters, $\varphi_1$ or $\varphi_2$.
Let $\mathcal{Q}^{(i)}$ be the maximal abelian pro-$p$ extension of $\Kan(\varphi_i)$ which is unramified away from $v(\Kan)$.
Rubin identifies $\mathcal{Q}^{(i)}$ with an anticyclotomic $p$-adic L-function, and Hida (also T. Finis) proves that the $\mu$-invariant of such an anticyclotomic $p$-adic L-function vanishes.
Combining their results, we deduce that $\mu(\mathcal{Q}^{(i)})=0$ for $i=1,2$.
Since $\widetilde{Y_i}$ is a quotient of $\mathcal{Q}^{(i)}$, it follows that $\mu(\widetilde{Y}_i)=0$ for $i=1,2$.
Hence, $\mu(Y_i)=0$ for $i=1,2$ as well.
\end{proof}
The most explicit examples of reducible Galois representations arise from elliptic curves with $p$-torsion points over the base field.
In this setting, the conditions of Theorem \ref{hidarubin} are not satisfied since $\{\varphi_1,\varphi_2\}=\{1, \overline{\chi}\}$.
We now state the main theorems of this article.
\begin{Th}\label{theorem1}
Suppose the following conditions hold.
\begin{enumerate}
\item $\closure{\rho}_E$ is reducible.
    \item Hypothesis $\ref{splithyp}$ is satisfied.
    \item If $\closure{\rho}_E$ is indecomposable, then Hypothesis $\ref{splithyp2}$ is satisfied.
    \item For $i=1,2$, the $\mu$-invariant $\mu(Y_i)=0$ (for instance, if the conditions of Theorem $\ref{hidarubin}$ are satisfied).
\end{enumerate}
Then, the fine Selmer group $\mathcal{R}_{p^{\infty}}(E/\Kan)$ is $\Lambda$-cotorsion with $\mu^{\op{fine}}(E/\Kan)=0$.
\end{Th}
This result is illustrated by Example $\ref{example1}$.
For elliptic curves over $\Q$ where hypotheses of Theorem \ref{hidarubin} are satisfied, a similar result is proven in \cite{CGLS20}. 

\begin{Remark} When all primes $v\in S\setminus S_p$ are split in $K$, the vanishing of the $\mu$-invariant for $Y_i$ is equivalent to the vanishing of the $\mu$-invariant of $\widetilde{Y}_i$ (see the proof of Proposition $\ref{Rvarphifinite2}$).
However, if a prime $v\in S\setminus S_p$ does not split in $K$, Fact $\ref{fact1}$ asserts that there are infinitely many primes above $v$. 
Hence, there are infinitely many splitting conditions cutting out $Y_i$ as a quotient of $\widetilde{Y}_i$.
The condition $\mu(Y_i)=0$ is therefore a more optimal condition than requiring $\mu(\widetilde{Y}_i)=0$.\end{Remark}
The next result is a (partial) converse to Theorem $\ref{theorem1}$.
\begin{Th}\label{theorem2}
Let $E$ be an elliptic curve defined over $K$.
Assume that
\begin{enumerate}
\item $\closure{\rho}_E$ is reducible.
    \item $\mathcal{R}_{p^{\infty}}(E/\Kan)$ is a cotorsion $\Lambda$-module with $\mu^{\op{fine}}(E/\Kan)=0$.
\end{enumerate}
If $\closure{\rho}_E$ is indecomposable, then $\mu(Y_1)=0$.
If $\closure{\rho}_E$ is split, then $\mu(Y_1)=\mu(Y_2)=0$.
\end{Th}
The next corollary takes note of a special case of interest.
\begin{Corollary}\label{corollory33}
Let $E$ be an elliptic curve defined over $K$ and $S$ be the set of primes dividing $\mathcal{N}p$, where $\mathcal{N}$ is the conductor of $E$.
Assume that
\begin{enumerate}
\item $E(K)[p]\neq 0$.
     \item $\mathcal{R}_{p^{\infty}}(E/\Kan)$ is a cotorsion $\Lambda$-module with $\mu^{\op{fine}}(E/\Kan)=0$.
     \item\label{corollory33c3} the Heegner hypothesis is satisfied: for $v\in S\setminus S_p$, let $l$ be the prime number such that $v|l$, then $l$ splits in $K$.
\end{enumerate}
Then the (classical) Iwasawa $\mu$-invariant, $\mu(\Kan/K)=0$.
\end{Corollary}
\begin{proof}
Note that $E(K)[p]\neq 0$ when the residual representation is of the form $\closure{\rho}_E\simeq \mtx{1}{\ast}{}{\overline{\chi}}$.
The result follows from Theorem $\ref{theorem2}$ and Proposition $\ref{Rvarphifinite2}$.
\end{proof}

\par Example $\ref{example2}$ illustrates Corollary $\ref{corollory33}$. Before proving the above results, a brief sketch of the method is provided.
The Kummer sequence \begin{equation}\label{kummersequence}0\rightarrow E[p]\rightarrow E[p^{\infty}]\xrightarrow{\times p} E[p^{\infty}]\rightarrow 0\end{equation}induces a comparison map 
\begin{equation}\label{psidef}\Psi:\mathcal{R}(E[p]/\Kan)\rightarrow \mathcal{R}_{p^{\infty}}(E/\Kan)[p].\end{equation}
The main theorem is deduced from the following statements.
\begin{enumerate}
    \item The kernel of $\Psi$ is finite.
    If Hypothesis $\ref{splithyp}$ is satisfied, then $\op{cok} \Psi$ is also finite.
    This step does not require any assumption on $\closure{\rho}_E$.
    \item Suppose that $\closure{\rho}_E$ is reducible.
    Further, if the conditions of Theorem $\ref{theorem1}$ are satisfied, then the mod-$p$ fine Selmer group $\mathcal{R}(E[p]/\Kan)$ is finite.
\end{enumerate}

We now analyze the map $\Psi$ and deduce a criterion for the vanishing of the fine $\mu$-invariant.
To show that $\op{cok}\Psi$ is finite when Hypothesis $\ref{splithyp}$ is satisfied, we prove the following result.

\begin{Lemma}\label{lemma31}
Let $v\in S\setminus\Sigma$ be a prime such that $v\nmid p$ and $\eta$ be a prime of $\Kan$ above $v$.
Then, the group $E(\Kan_{\eta})[p^{\infty}]$ is $p$-divisible.
\end{Lemma}

\begin{proof}
Since $v\notin \Sigma$, (at least) one of the following conditions is satisfied:
 \begin{enumerate}
 \item $v\nmid (\mathcal{N}/\closure{\mathcal{N}})$,
  \item $p\geq 5$, $\mu_p \subset K_v$, and $E$ has non-split multiplicative reduction or additive reduction at $v$,
  \item $p=3$, $\mu_3\subset K_v$, and $E$ has non-split multiplicative reduction at $v$.
\end{enumerate}
The claim follows from the proof of \cite[Lemma 4.1.2]{EPW} and \cite[Proposition 5.1 (\rm{iii})]{HM99}.
\end{proof}
For $v\in S$, let $h_v$ denote the natural map
\[h_v:\mathcal{H}_v( \Kan, E[p])\rightarrow \mathcal{H}_v(\Kan, E[p^{\infty}])[p]\]induced from $\eqref{kummersequence}$.
For each prime $\eta|v$ of $\Kan$, set $h_{\eta}$ to denote the natural map \[h_{\eta}: H^1(\Kan_{\eta}, E[p])\rightarrow H^1(\Kan_{\eta}, E[p^{\infty}])[p],\] also induced from $\eqref{kummersequence}$.
\begin{Corollary}\label{cor32}
If Hypothesis $\ref{splithyp}$ is satisfied, then $\op{ker}h_v$ is finite for $v\in S$.
\end{Corollary}
\begin{proof}
The kernel of $h_v$ is the product of $\ker h_{\eta}$, as $\eta$ ranges over $v(\Kan)$.
By Kummer theory, $\ker h_{\eta}$ is isomorphic to $H^0(\Kan_{\eta}, E[p^{\infty}])/p$; hence, it is finite.

First, consider the case when $v\in \Sigma\cup S_p$.
Let $l$ be the prime number for which $v|l$.
By Hypothesis $\ref{splithyp}$, if $v\in \Sigma$, then $l$ is split in $K$.
Therefore, by Fact $\ref{fact1}$, the set of primes $v(\Kan)$ is finite for $v\in \Sigma\cup S_p$.
Since there are only finitely many primes $\eta\in v(\Kan)$, $\op{ker} h_v$ is finite in this case.

Next, consider the case when $l\in S\setminus (\Sigma\cup S_p)$.
Lemma $\ref{lemma31}$ asserts that the group $E(\Kan_{\eta})[p^{\infty}]$ is $p$-divisible for all $\eta\in v(\Kan)$.
In this case,
\[\op{ker} h_{\eta}=H^0(\Kan_{\eta}, E[p^{\infty}])/p=0.\]
Therefore, $h_v$ is injective.
\end{proof}

\begin{Proposition}\label{prop33}
The kernel of $\Psi$ is finite.
If Hypothesis $\ref{splithyp}$ is satisfied, then $\op{cok} \Psi$ is also finite.
\end{Proposition}
\begin{proof}
The Kummer sequence $\eqref{kummersequence}$ induces the commutative diagram:
\[
{\footnotesize\begin{tikzcd}[column sep = small, row sep = large]
0\arrow{r} & \mathcal{R}(E[p]/\Kan) \arrow{r}\arrow{d}{\Psi} & H^1(K_S/\Kan, E[p])\arrow{r} \arrow{d}{g} & \op{im}(\closure{\Phi}_E)\arrow{r} \arrow{d}{h} & 0\\
0\arrow{r} & \mathcal{R}_{p^{\infty}}(E/\Kan)[p] \arrow{r} & H^1(K_S/\Kan, E[p^{\infty}])[p] \arrow{r}  &\bigoplus_{v\in S} \mathcal{H}_v(\Kan, E[p^{\infty}])[p].
\end{tikzcd}}
\]Here, $\closure{\Phi}_E$ is the natural map
\[\closure{\Phi}_E: H^1(K_S/\Kan, E[p])\rightarrow \bigoplus_{v\in S} \mathcal{H}_v(\Kan, E[p]).\] 
Clearly, the map $g$ is surjective.
The snake lemma yields an exact sequence,
\begin{equation}\label{ses}0\rightarrow \op{ker} \Psi \rightarrow \op{ker} g\rightarrow \op{ker} h \rightarrow \op{cok} \Psi \rightarrow 0.\end{equation} Since $\ker g \simeq H^0(\Kan, E[p^{\infty}])/p$, it is finite (with cardinality at most $p^2$).
From $\eqref{ses}$, we deduce that $\op{ker}\Psi$ is finite.
It follows from Corollary $\ref{cor32}$ that $\ker h$ is finite.
Therefore, $\op{cok}\Psi$ is finite as well.
\end{proof}
\begin{Remark}
The proof of the above proposition shows that $\ker h$ must be finite if $\mu^{\op{fine}}(E/\Kan)=0$ (even if Hypothesis $\ref{splithyp}$ is not satisfied).
This point is of considerable interest to the authors, since there is no apparent reason to suggest why this should be true in general.
\end{Remark}
\begin{Proposition}\label{Rrhobarfinite}
Let $E_{/K}$ be an elliptic curve.
The following assertions hold.
\begin{enumerate}
    \item If $\mathcal{R}_{p^\infty}(E/\Kan)$ is $\Lambda$-cotorsion with $\mu^{\op{fine}}(E/\Kan)=0$, then $\mathcal{R}(E[p]/\Kan)$ is finite.
    \item Assume that Hypothesis $\ref{splithyp}$ is satisfied.
    Then, the fine Selmer group is $\Lambda$-cotorsion with $\mu^{\op{fine}}(E/\Kan)=0$ if and only if $\mathcal{R}(E[p]/\Kan)$ is finite.
\end{enumerate}
\end{Proposition}
\begin{proof}
The assertion is a consequence of Lemma $\ref{lemma27}$ and Proposition $\ref{prop33}$.
\end{proof}

Let $F=K(E[p])$ be the field extension of $K$ left fixed by the kernel of $\closure{\rho}_E$.
Note that $F$ is a Galois extension of $K$ with Galois group isomorphic to the image of $\closure{\rho}_E$.
In particular, if $\closure{\rho}_E$ is reducible, then it is a $2$-step solvable extension of $K$.
Set $F_\infty$ to denote the compositum $\Kan\cdot F$.
In the remainder of this section, let $\mathcal{X}_E^S$ be the $S$-Hilbert class field extension of $F_\infty$. %i.e., the maximal abelian pro-$p$ extension which is unramified at all primes of $F_{\infty}$ and in which all primes of $S\left(F_{\infty}\right)$ split completely.
Note that there is a natural action of $\op{Gal}(F/K)$ on $\mathcal{X}_E^S$.
The $\mu$-invariant of $\mathcal{X}_E^S$ is taken with respect to the $\Z_p$-extension $F_\infty/F$.
Recall that in this section, there is no hypothesis on (the reducibility) of the residual representation $\closure{\rho}_E$.

\begin{comment}
Let $K(E[p])$ be the field extension of $K$ left fixed by the kernel of $\closure{\rho}_E$.
Note that $K(E[p])$ is a Galois extension of $K$ with Galois group isomorphic to the image of $\closure{\rho}_E$.
In particular, if $\closure{\rho}_E$ is reducible, then it is a $2$-step solvable extension of $K$.
Set $\Kan(E[p])$ to denote the compositum $\Kan\cdot K(E[p])$. 
In the remainder of this section, let $\mathcal{X}_E^S$ be the $S$-Hilbert class field extension of $\Kan(E[p])$, i.e., the maximal abelian pro-$p$ extension which is unramified at all primes of $\Kan(E[p])$ and in which all primes of $S\left(\Kan(E[p])\right)$ split completely. Note that there is a natural of $\op{Gal}(K(E[p])/K)$ on $\mathcal{X}_E^S$. The $\mu$-invariant of $\mathcal{X}_E^S$ is taken with respect to the $\Z_p$-extension $\Kan(E[p])/K(E[p])$. Recall that in this section, there is no hypothesis on (the reducibility) of the residual representation $\closure{\rho}_E$.

\end{comment}
\begin{Corollary}
Let $E_{/K}$ be an elliptic curve such that
\begin{enumerate}
    \item Hypothesis $\ref{splithyp}$ is satisfied.
    \item $\mu(\mathcal{X}_E^S)=0$.
\end{enumerate}Then, the fine $\mu$-invariant, $\mu^{\op{fine}}(E/\Kan)=0$.
\end{Corollary}

\begin{proof}
By inflation-restriction, we have 
{\footnotesize\[
0 \rightarrow H^1\left(\op{Gal}(F_\infty/\Kan), E[p]\right) \rightarrow H^1(\Kan, E[p])\rightarrow \op{Hom}\left(F, E[p]\right)^{\op{Gal}(F/K)}.
\]}
This induces the following map with finite kernel
\begin{equation}\label{inflationrestriction}\mathcal{R}(E[p]/\Kan)\rightarrow \op{Hom}\left(\mathcal{X}_E^S/p, E[p]\right)^{\op{Gal}(F/K)}.\end{equation} %H^1\left(\op{Gal}(\Kan(E[p])/\Kan), E[p]\right) is finite
Since $\mu(\mathcal{X}_E^S)=0$, the group $\mathcal{X}_E^S/p$ is finite.
By $\eqref{inflationrestriction}$, the mod-$p$ fine Selmer group $\mathcal{R}(E[p]/\Kan)$ is finite.
The result follows from Proposition $\ref{Rrhobarfinite}$.
\end{proof}

\section{Finiteness of the residual fine Selmer group}
\label{finiteness of the residual fine Selmer group}
Throughout this section $\closure{\rho}_E$ is assumed to be reducible.
Therefore, $E[p]$ fits into the short exact sequence, \[0\rightarrow \F_p(\varphi_1)\rightarrow E[p]\rightarrow \F_p(\varphi_2)\rightarrow 0.\]
This allows the analysis of the algebraic structure of the mod-$p$ fine Selmer group $\mathcal{R}(E[p]/\Kan)$ in terms of the mod-$p$ fine Selmer groups $\mathcal{R}(\mathbb{F}_p(\varphi_i)/\Kan)$ for $i=1,2$ (definitions given below).
This completes the proof of Theorems $\ref{theorem1}$ and $\ref{theorem2}$.

\begin{Definition}
Let $S$ be the set of primes dividing $\mathcal{N}p$, where $\mathcal{N}$ is the conductor of $E$.
The fine Selmer group $\mathcal{R}(\F_p(\varphi_i)/ \Kan)$ is defined as follows
\[\mathcal{R}(\F_p(\varphi_i)/ \Kan):=\op{ker}\left\{H^1(K_S/\Kan, \F_p(\varphi_i))\rightarrow \prod_{\eta\in S(\Kan)}H^1(\Kan_{\eta}, \F_p(\varphi_i))\right\}.\]
\end{Definition}
\noindent The dependence on $S$ is suppressed in the notation.
\begin{Proposition}\label{Rvarphifinite}Suppose that $\closure{\rho}_E$ is reducible and let $\varphi_i$ be one of the characters on the diagonal of $\closure{\rho}_E$.
Then, $\mathcal{R}(\F_p(\varphi_i)/\Kan)$ is finite if and only if $\mu(Y_i)=0$.
\end{Proposition}
\begin{proof}
Set $\Delta:=\op{Gal}(K(\varphi_i)/K)$.
Since the order of $\Delta$ is coprime to $p$, 
\[\op{Gal}(\Kan(\varphi_i)/K)\simeq \Delta\times \op{Gal}(\Kan/K).\] 
The group $\op{Gal}(\Kan(\varphi_i)/K)$ acts on $\mathcal{X}^{(i)}$.
Let $g\in \op{Gal}(\Kan(\varphi_i)/K)$ and $x\in \mathcal{X}^{(i)}$. 
Choose a lift $\tilde{g}\in \op{Gal}(\mathcal{L}^{(i)}/K)$, and set $g\cdot x:=\tilde{g} x\tilde{g}^{-1}$.
Since $\mathcal{X}^{(i)}$ is abelian, $g\cdot x$ is independent of the choice of lift, $\tilde{g}$.
This induces the action of $\Delta$ on $\mathcal{X}^{(i)}$.
There is a natural isomorphism
\[\op{Hom}(Y_i,\F_p)\simeq \op{Hom}(\mathcal{X}^{(i)}, \F_p(\varphi_i))^{\Delta}.\]
Since the order of $\op{Gal}(\Kan(\varphi_i)/\Kan)$ is prime to $p$, it follows (from restriction-corestriction) that the cohomology group $H^j(\op{Gal}(\Kan(\varphi_i)/\Kan),\F_p(\varphi_i))=0$ for $j>0$. %Note that $\F_p(\varphi_i)^{\op{Gal}(K_S/\Kan(\varphi_i))}=\F_p(\varphi_i)$.
It follows from the inflation-restriction sequence that 
\[H^1(K_S/\Kan, \F_p(\varphi_i))\xrightarrow{\sim}\op{Hom}(\op{Gal}(K_S/\Kan(\varphi_i)), \F_p(\varphi_i))^{\Delta},\]
where $\Delta$ is identified with $\op{Gal}(\Kan(\varphi_i)/\Kan)$.
This induces an isomorphism
\[\iota: \mathcal{R}(\F_p(\varphi_i)/\Kan)\xrightarrow{\sim} \op{Hom}(Y_i/pY_i,\F_p).\]
It is an easy consequence of the structure theory of $\Lambda$-modules (see the argument in the proof of Lemma $\ref{lemma27}$) that the $\mu$-invariant of $Y_i$ is zero if and only if $Y_i/pY_i$ is finite.
The result follows.
\end{proof}
\begin{Proposition}\label{Rvarphifinite2}
Suppose that $\closure{\rho}_E$ is reducible and $\varphi_i$ is one of the characters on the diagonal.
Suppose that the following assumptions hold.
\begin{enumerate}
    \item $\mathcal{R}(\F_p(\varphi_i)/\Kan)$ is finite.
    \item The Heegner hypothesis is satisfied: let $v\in S\setminus S_p$ and $l$ be the prime number such that $v|l$, then $l$ splits in $K$.
\end{enumerate}
    Then, we have that $\mu(\widetilde{Y}_i)=0$.
\end{Proposition}
This result is used in the proof of Corollary $\ref{corollory33}$ where it is further assumed that $\varphi_i=1$.
In this special case, the assertion is that the (classical) Iwasawa invariant $\mu(\Kan/K)$ vanishes.
\begin{proof}
Recall that the Galois extensions $\mathcal{M}^{(i)}$ and $\mathcal{X}^{(i)}$ of $\Kan(\varphi_i)$ are defined on p. 7.
Consider the natural action of $\Delta:=\op{Gal}(K(\varphi_i)/K)$ on $\mathcal{M}^{(i)}$ (resp. $\mathcal{X}^{(i)}$) and recall that $\widetilde{Y}_i$ (resp. $Y_i$) denotes $\varphi_i$-component of $\mathcal{M}^{(i)}$ (resp. $\mathcal{X}^{(i)})$ with respect to this action.
At each prime $\eta$ of $\Kan$, set $H^1_{\op{nr}}(\Kan_{\eta}, \F_p(\varphi_i))$ to denote the subspace of unramified cohomology classes in $H^1(\Kan_{\eta}, \F_p(\varphi_i))$.
Set $\mathcal{R}'(\F_p(\varphi_i)/\Kan)$ to consist of cohomology classes $f\in H^1(K_S/\Kan, \F_p(\varphi_i))$ which are unramified at every prime $\eta\in S(\Kan)$.
By an application of inflation-restriction (see the argument in the proof of Proposition $\ref{Rvarphifinite}$), 
\[\begin{split}&\mathcal{R}'(\F_p(\varphi_i)/\Kan)\simeq \op{Hom}(\mathcal{M}^{(i)}/p, \F_p(\varphi_i))^{\Delta}\simeq \op{Hom}(\widetilde{Y}_i/p, \F_p),\\&\mathcal{R}(\F_p(\varphi_i)/\Kan)\simeq \op{Hom}(\mathcal{X}^{(i)}/p, \F_p(\varphi_i))^{\Delta}\simeq \op{Hom}(Y_i/p, \F_p).\end{split}\]
The groups fit into an exact sequence,
\[0\rightarrow \mathcal{R}(\F_p(\varphi_i)/\Kan)\rightarrow \op{Hom}(\widetilde{Y}_i/p, \F_p(\varphi_i)) \rightarrow \prod_{\eta\in S(\Kan)} H^1_{\op{nr}}(\Kan_{\eta}, \F_p(\varphi_i)).\]
By Fact $\ref{fact1}$, the set $S(\Kan)$ is finite. 
Hence, the group on the right is finite. We deduce that $\widetilde{Y}_i/p$ is finite, therefore the $\mu$-invariant of $\widetilde{Y}_i$ is zero.
\end{proof}

\begin{proof}[Proof of Theorem $\ref{theorem1}$]
\par By Proposition $\ref{Rrhobarfinite}$, it suffices to show that $\mathcal{R}(E[p]/\Kan)$ is finite.
The short exact sequence \[0\rightarrow \F_p(\varphi_1)\rightarrow E[p] \rightarrow \F_p(\varphi_2)\rightarrow 0,\] induces the long exact sequence
\[\cdots\rightarrow H^0(\Kan_{\eta}, \F_p(\varphi_2))\xrightarrow{\delta_{\eta}^0} H^1(\Kan_{\eta}, \F_p(\varphi_1))\rightarrow H^1(\Kan_{\eta}, E[p])\rightarrow \cdots .\]
Set $S_{\op{ac}}:=S(\Kan)$ and \[U:=\prod_{\eta\in S_{\op{ac}}}\op{image} \delta_{\eta}^0 \subset \prod_{\eta\in S_{\op{ac}}} H^1(\Kan_{\eta}, \F_p(\varphi_1)).\]
If $\closure{\rho}_E$ is split, then $\delta_{\eta}^0$ is the 0 map; hence, $U=0$. \newline
Next, consider the case when $\closure{\rho}_E$ is indecomposable.
When $v\in S\setminus (\Sigma(\varphi_2)\cup S_p)$, note that $H^0(K_v,\F_p(\varphi_2))=0 $.
Let $\eta\in v(\Kan)$.
Since $\Kan_{\eta}/K_v$ is a $p$-extension, for all $v\in S\setminus (\Sigma(\varphi_2)\cup S_p)$, the product $\prod_{\eta\in v(\Kan)}H^0(\Kan_{\eta},\F_p(\varphi_2))=0$.
On the other hand, for all primes $v\in \Sigma(\varphi_2)\cup S_p$, it follows from Hypothesis $\ref{splithyp2}$ and Fact $\ref{fact1}$ that there are only finitely many primes $\eta$ in $v(\Kan)$.
Therefore, the assumptions imply that $U$ is finite.
\par Consider the diagram
\[
{\footnotesize \begin{tikzcd}
 & H^1(K_S/\Kan, \F_p(\varphi_1)) \arrow{d}{\alpha} \arrow{r} & H^1(K_S/\Kan, E[p]) \arrow{r} \arrow{d}{\beta} & H^1(K_S/\Kan, \F_p(\varphi_2)) \arrow{d}{\gamma}\\
0\arrow{r}& (\displaystyle{\prod_{\eta\in S_{\op{ac}}}} H^1(\Kan_{\eta}, \F_p(\varphi_1)))/U\arrow{r}  &  \displaystyle{\prod_{\eta\in S_{\op{ac}}}} H^1(\Kan_{\eta}, E[p]) \arrow{r} & \displaystyle{\prod_{\eta\in S_{\op{ac}}}} H^1(\Kan_{\eta}, \F_p(\varphi_2))
\end{tikzcd}}
\]
and the associated short exact sequence
\begin{equation}\label{es1}\op{ker} \alpha \rightarrow \mathcal{R}(E[p]/\Kan)\rightarrow \mathcal{R}(\F_p(\varphi_2)/\Kan).\end{equation}
The group $\op{ker}\alpha$ fits into an exact sequence 
\begin{equation}\label{es2}0\rightarrow \mathcal{R}(\F_p(\varphi_1)/\Kan)\rightarrow \op{ker}\alpha \rightarrow U.\end{equation}
Proposition $\ref{Rvarphifinite}$ asserts that $\mathcal{R}(\F_p(\varphi_i), \Kan)$ are finite for $i=1,2$.
This requires the assumption that $\mu(Y_i)=0$ for $i=1,2$.
It follows from $\eqref{es1}$ and $\eqref{es2}$ that $\mathcal{R}(E[p]/\Kan)$ is also finite.
By Lemma $\ref{lemma27}$, we have that $\mathcal{R}_{p^\infty}(E/\Kan)$ is $\Lambda$-cotorsion with $\mu^{\op{fine}}(E/\Kan)=0$.
This completes the proof.
\end{proof}

\begin{proof}[Proof of Theorem $\ref{theorem2}$]
\par Refer to the argument in the proof of Theorem $\ref{theorem1}$.
Lemma $\ref{lemma27}$ asserts that if $\mathcal{R}_{p^\infty}(E/\Kan)$ is $\Lambda$-torsion with $\mu^{\op{fine}}(E/\Kan)=0$, then $\mathcal{R}(E[p]/\Kan)$ is finite.
Refer to $\eqref{es1}$.
Since $H^0(K_S/\Kan, \F_p(\varphi_1))$ is finite, finiteness of $\mathcal{R}(E[p]/\Kan)$ implies that $\op{ker}\alpha$ is finite.
By $\eqref{es2}$, it is seen that $\mathcal{R}(\F_p(\varphi_1)/\Kan)$ is finite.
Therefore, by Proposition $\ref{Rvarphifinite}$, it follows that $\mu(Y_1)=0$.
This completes the proof in the case when $\closure{\rho}_E$ is indecomposable.
When $\closure{\rho}_E$ is split, the proof is completed by interchanging the roles of $\varphi_1$ and $\varphi_2$.
\end{proof}

\section{Congruent Galois representations}
\label{Congruent Galois representations}
%\DK{Is it possible to show that $\mu^{\op{fine}}(E/\Kan)=0$ iff $H^2(G_S(\Kan), E[p])=0$?}
\par In this section, we consider two elliptic curves $E_1$ and $E_2$ defined over $K$ which are $p$-congruent, i.e. their residual representations are isomorphic,
\[\closure{\rho}_{E_1}\simeq \closure{\rho}_{E_2}.\]
We make no assumption on the reducibility of the residual representations.
However, it is assumed that both $E_1$ and $E_2$ satisfy the Heegner hypothesis.
\begin{hypothesis}[Heegner hypothesis]
Let $E_{/K}$ be an elliptic curve with conductor $\mathcal{N}$, and $S$ be the set of primes dividing $\mathcal{N}p$.
Let $v\in S$ and $l$ be the prime number such that $v|l$.
Then, $l$ splits in $K$ if $l\neq p$.
\end{hypothesis}
\noindent We prove the following main result in this section.
\begin{Th}\label{theorem3}
Let $E_1$ and $E_2$ be elliptic curves over $K$. 
Suppose that the following assumptions hold. 
\begin{enumerate}
    \item Both $E_1$ and $E_2$ satisfy the Heegner hypothesis.
    \item The residual Galois representations, $\closure{\rho}_{E_1}$ and $\closure{\rho}_{E_2}$, are isomorphic.
\end{enumerate}
Then, 
\[\begin{split}& \mathcal{R}_{p^{\infty}}(E_1/\Kan) \text{ is }\Lambda\text{-cotorsion with }\mu^{\op{fine}}(E_1/\Kan)=0\\ \Leftrightarrow & \mathcal{R}_{p^{\infty}}(E_2/\Kan) \text{ is }\Lambda\text{-cotorsion with }\mu^{\op{fine}}(E_2/\Kan)=0.\end{split}\]
\end{Th}
The key result used in the proof of the above theorem will be Proposition $\ref{Rrhobarfinite}$, which states that \[\mathcal{R}_{p^{\infty}}(E_i/\Kan) \text{ is }\Lambda\text{-cotorsion with }\mu^{\op{fine}}(E_i/\Kan)=0\Leftrightarrow \mathcal{R}(E_i[p]/\Kan)\text{ is finite.}\]
In view of Lemma \ref{lemma27}, it suffices to prove the following result.
\begin{Proposition}\label{lastprop}
Let $E_1$ and $E_2$ be elliptic curves satisfying the conditions of Theorem $\ref{theorem3}$.
Then $\mathcal{R}(E_1[p]/\Kan)$ is finite if and only if $\mathcal{R}(E_2[p]/\Kan)$ is finite.
\end{Proposition}
\begin{proof}
Let $\mathcal{N}_i$ be the conductor of $E_i$ and $S_i$ be the set of primes dividing $\mathcal{N}_i p$.
Denote by $T$ the union $S_1\cup S_2$.
Since $E_1[p]$ and $E_2[p]$ are isomorphic as Galois modules, the mod-$p$ fine Selmer groups are isomorphic, i.e.
\[\mathcal{R}^T(E_1[p]/\Kan) \simeq \mathcal{R}^T(E_2[p]/\Kan)\] (recall $\mathcal{R}^T(E_i[p]/\Kan)$ from Definition $\ref{Tfine}$).
Let $v\in T$ and $\eta\in v(\Kan)$.
Denote by $\op{G}_{\eta}$ the absolute Galois group of $\Kan_{\eta}$ and by $\op{I}_{\eta}\subset \op{G}_{\eta}$ its inertia subgroup.
Set $H^1_{\op{nr}}(\Kan_{\eta}, E_i[p])$ to denote the kernel of the restriction map,
\[\op{res}: H^1(\Kan_{\eta}, E_i[p])\longrightarrow H^1(\op{I}_{\eta}, E_i[p])^{\op{G}_{\eta}/\op{I}_{\eta}}.\]
It follows from inflation-restriction that $H^1_{\op{nr}}(\Kan_{\eta}, E_i[p])$ can be identified with $H^1(\op{G}_{\eta}/\op{I}_{\eta}, E_i[p]^{\op{I}_{\eta}})$. %Rubin's refernce if required 
Since $\op{G}_{\eta}/\op{I}_{\eta}\simeq \hat{\Z}$, it follows that (see \cite[Proposition 1.7.7]{NSW08})
\[H^1(\op{G}_{\eta}/\op{I}_{\eta}, E_i[p]^{\op{I}_{\eta}})\simeq H_0(\op{G}_{\eta}/\op{I}_{\eta}, E_i[p]^{\op{I}_{\eta}})\] and hence is finite.
By the Heegner hypothesis and Fact $\ref{fact1}$, the set $v(\Kan)$ is finite for $v\in T$.
Let $\mathcal{H}_{v}^{\op{nr}}(\Kan, E[p])\subseteq \mathcal{H}_{v}(\Kan, E[p])$ denote the sum,
\[\mathcal{H}_{v}^{\op{nr}}(\Kan, E_i[p]):=\bigoplus_{\eta\in v(\Kan)}H^1_{\op{nr}}(\Kan_{\eta}, E_i[p]).\]
It follows from the above discussion that $\mathcal{H}_{v}^{\op{nr}}(\Kan, E_i[p])$ is finite for $v\in T$.

There is a natural exact sequence,
\[0\rightarrow \mathcal{R}^T(E_i[p]/\Kan)\rightarrow \mathcal{R}(E_i[p]/\Kan)\rightarrow \bigoplus_{v\in T\setminus S_i} \mathcal{H}_{v}^{\op{nr}}(\Kan, E_i[p]).\]
Since the last term in the above sequence is finite, it follows that $\mathcal{R}(E_i[p]/\Kan)$ is finite if and only if $\mathcal{R}^T(E_i[p]/\Kan)$ is.
The result follows from this.
\end{proof}
\begin{proof}[Proof of Theorem $\ref{theorem3}$]
Note that the Heegner hypothesis implies Hypothesis $\ref{splithyp}$. Therefore, Proposition $\ref{Rrhobarfinite}$ applies in this setting and asserts that \[\mathcal{R}_{p^{\infty}}(E_i/\Kan) \text{ is }\Lambda\text{-cotorsion with }\mu^{\op{fine}}(E_i/\Kan)=0\Leftrightarrow \mathcal{R}(E_i[p]/\Kan)\text{ is finite.}\] On the other hand, Proposition $\ref{lastprop}$ states that
\[\mathcal{R}(E_1[p]/\Kan)\text{ is finite }\Leftrightarrow \mathcal{R}(E_2[p]/\Kan)\text{ is finite.}\] The result follows.
\end{proof}

\section{Examples}
\label{examples}
In this section, we list some examples to illustrate the results in the article.
\begin{Example}\label{example1}We begin with an example which illustrates Theorem $\ref{theorem1}$.
In this example, $p=3$ and $E$ is the elliptic curve 81.1-CMa1 over $K = \Q(\sqrt{-3})$.
This is the elliptic curve with Weierstass equation $y^2+y=x^3$ base changed to $K$.
It has complex multiplication and its endomorphism ring is $\Z[\frac{1+\sqrt{-3}}{2}]$.
The group $E(K)[3]$ is the full group $E[3]$ (and isomorphic to $\Z/3 \times \Z/3$).
In particular, residual representation $\closure{\rho}_E$ is the trivial representation $\mtx{1}{}{}{1}$.
Therefore, $\closure{\rho}_E$ is split, $\varphi_1=\varphi_2=1$, and $Y_1$ and $Y_2$ are both the $3$-Hilbert class field extension of $\Kan$.
The prime $p=3$ is a prime of additive reduction and 3 is ramified in $K$.
The conductor of $E$ is $\mathcal{N}=9\mathcal{O}_K$ and in particular is not divisible by any primes $v\nmid 3$.
It follows that $\Sigma=\emptyset$ and Hypothesis $\ref{splithyp}$ is satisfied.
This unique prime above $(3)$ is totally ramified in the anticyclotomic $\Z_3$-extension of $K$.
Since the class number of $K$ is $1$, we know that $\mu(\Kan/K)=0$ \cite[Proposition 4.2]{Och09}.
It follows from Theorem $\ref{theorem1}$ that $\mathcal{R}_{3^{\infty}}(E/\Kan)$ is a cotorsion $\Lambda$-module with  $\mu^{\op{fine}}(E/\Kan)=0$ (at $p=3$).
In other words, $\mathcal{R}_{3^{\infty}}(E/\Kan)$ is a cofinitely generated $\Z_3$-module.
\end{Example}

\begin{Example}\label{example2} This example demonstrates Corollary $\ref{corollory33}$.
Let $K=\Q(\sqrt{-10})$, $p=5$, and $E$ be the elliptic curve $11a1$ over $\Q$ defined by the Weierstrass equation $y^2+y=x^3-x^2-10x-20$.
The $5$-torsion group $E(\Q)[5]$, is equal to $\Z/5\Z$.
In particular, $E(K)[5]\neq 0$.
Assume that $\mathcal{R}_{5^{\infty}}(E/\Kan)$ is a cotorsion $\Lambda$-module.
The set $S$ consists of the primes above $11$ in $K$.
Since $11$ splits in $K$, it follows that condition $\eqref{corollory33c3}$ of Corollary $\ref{corollory33}$ is satisfied.
Corollary $\ref{corollory33}$ asserts that if the $\mu$-invariant of the fine Selmer group $\mathcal{R}_{5^{\infty}}(E/\Kan)$ is zero, then the Iwasawa $\mu$-invariant of the $5$-Hilbert class field extension of $\Kan/K$ is zero as well.
\end{Example}

\begin{Example}\label{example3} 
This example demonstrates an application of Theorem $\ref{theorem3}$.
Consider the elliptic curves $E_1 = 201c1$ and $E_2 = 469a1$. 
Both elliptic curves have rank $1$ with good ordinary reduction at the prime $p=5$.
Further, there is an isomorphism of $\op{G}_{\Q}$-modules, $E_1[5] \simeq E_2[5]$.
Note that $E_1$ has bad reduction at the primes $3$, $67$ and $E_2$ has bad reduction at the primes $7$, $67$.
Set $M$ to be the product of primes of bad reduction $3\cdot 7 \cdot 67$.
The primes of bad reduction of $E_1, E_2$ split in infinitely many imaginary quadratic number field $K$ of the form $\Q(\sqrt{-Mk + 1})$ where $k\in \Z_{\geq 1}$, i.e. the Heegner hypothesis is satisfied.
Theorem $\ref{theorem3}$ applies; hence
\[\begin{split}& \mathcal{R}_{p^{\infty}}(E_1/\Kan) \text{ is }\Lambda\text{-cotorsion with }\mu^{\op{fine}}(E_1/\Kan)=0\\ \Leftrightarrow & \mathcal{R}_{p^{\infty}}(E_2/\Kan) \text{ is }\Lambda\text{-cotorsion with }\mu^{\op{fine}}(E_2/\Kan)=0.\end{split}\]
\end{Example}

\begin{Example}\label{example4} 
Through this final example we demonstrate that Theorem $\ref{theorem3}$ can be used to relate $\mu$-invariants of isogenous elliptic curves over $\Kan$.
Consider the isogenous elliptic curves $E_1 = 17a1$ and $E_2 = 17a2$. 
Both elliptic curves have rank $0$ and good supersingular reduction at the prime $p=11$. Further, there is an isomorphism of $\op{G}_{\Q}$-modules, $E_1[11] \simeq E_2[11]$.
Consider the imaginary quadratic number field, $K= \Q(\sqrt{-8})$.
Note that the prime 17 splits in $K$, i.e. both $E_1$ and $E_2$ satisfy the Heegner hypothesis.
It is known that $\mu^{\op{fine}}(E_1/\Kan)=0$ (see \cite[Table in \S 4]{Mat18}). It follows from Theorem \ref{theorem3} that $\mu^{\op{fine}}(E_2/\Kan)=0$, as well.
\end{Example}

\section*{Acknowledgement}
The first named author is funded by the PIMS Postdoctoral Fellowship. 
\bibliographystyle{abbrv}
\bibliography{references}

\begin{thebibliography}{10}

\bibitem{Ari14}
C.~S. Aribam.
\newblock On the $\mu$-invariant of fine {Selmer} groups.
\newblock {\em J. Number Theory}, 135:284--300, 2014.

\bibitem{BD05}
M.~Bertolini and H.~Darmon.
\newblock Iwasawa's main conjecture for elliptic curves over anticyclotomic
  $\mathbb{Z}_p$-extensions.
\newblock {\em Ann. Math.}, pages 1--64, 2005.

\bibitem{Bri07}
D.~Brink.
\newblock Prime decomposition in the anti-cyclotomic extension.
\newblock {\em Math. Comp.}, 76(260):2127--2138, 2007.

\bibitem{CGLS20}
F.~Castella, G.~Grossi, J.~Lee, and C.~Skinner.
\newblock On the anticyclotomic iwasawa theory of rational elliptic curves at
  eisenstein primes.
\newblock {\em arXiv preprint arXiv:2008.02571}, 2020.

\bibitem{CS05}
J.~Coates and R.~Sujatha.
\newblock Fine {Selmer} groups of elliptic curves over $p$-adic {Lie}
  extensions.
\newblock {\em Math. Ann.}, 331(4):809--839, 2005.

\bibitem{Dri02}
M.~J. Drinen.
\newblock Finite submodules and {Iwasawa} $\mu$-invariants.
\newblock {\em J. Number Theory}, 93(1):1--22, 2002.

\bibitem{Dri03}
M.~J. Drinen.
\newblock Iwasawa $\mu$-invariants of elliptic curves and their symmetric
  powers.
\newblock {\em J. Number Theory}, 102(2):191--213, 2003.

\bibitem{EPW}
M.~Emerton, R.~Pollack, and T.~Weston.
\newblock Variation of {I}wasawa invariants in {Hida} families.
\newblock {\em Invent. Math.}, 163(3):523--580, 2006.

\bibitem{FW79}
B.~Ferrero and L.~C. Washington.
\newblock The {Iwasawa} invariant $\mu_p$ vanishes for abelian number fields.
\newblock {\em Ann. Math.}, pages 377--395, 1979.

\bibitem{Fin06}
T.~Finis.
\newblock The $\mu$-invariant of anticyclotomic ${L}$-functions of imaginary
  quadratic fields.
\newblock {\em Journal f{\"u}r die reine und angewandte Mathematik},
  2006(596):131--152, 2006.

\bibitem{Gre89}
R.~Greenberg.
\newblock Iwasawa theory for $p$-adic representations.
\newblock In {\em Algebraic Number Theory—in Honor of K. Iwasawa}, pages
  97--137. Mathematical Society of Japan, 1989.

\bibitem{Gre99}
R.~Greenberg.
\newblock Iwasawa theory for elliptic curves.
\newblock In {\em Arithmetic theory of elliptic curves (Cetraro, 1997)}, volume
  1716, pages 51--144. Springer, 1999.

\bibitem{Gre11}
R.~Greenberg.
\newblock {\em Iwasawa theory, projective modules, and modular
  representations}.
\newblock American Math. Soc., 2010.

\bibitem{GV00}
R.~Greenberg and V.~Vatsal.
\newblock On the {Iwasawa} invariants of elliptic curves.
\newblock {\em Invent. Math.}, 142(1):17--63, 2000.

\bibitem{HM99}
Y.~Hachimori and K.~Matsuno.
\newblock An analogue of {Kida's} formula for the {Selmer} groups of elliptic
  curves.
\newblock {\em J. Alg. Geom.}, 8:581--601, 1999.

\bibitem{Hid10}
H.~Hida.
\newblock The {I}wasawa $\mu$-invariant of $p$-adic {H}ecke ${L}$-functions.
\newblock {\em Annals of mathematics}, pages 41--137, 2010.

\bibitem{Iwa73}
K.~Iwasawa.
\newblock On the $\mu$-invariants of $\mathbb{Z}_\ell$-extensions, number
  theory.
\newblock {\em Algebraic Geometry and Commutative Algebra}, pages 1--11, 1973.

\bibitem{JS11}
S.~Jha and R.~Sujatha.
\newblock On the {Hida} deformations of fine {Selmer} groups.
\newblock {\em J. Algebra}, 338(1):180--196, 2011.

\bibitem{Kat04}
K.~Kato.
\newblock $p$-adic {Hodge} theory and values of zeta functions of modular
  forms.
\newblock {\em Ast{\'e}risque}, 295:117--290, 2004.

\bibitem{Kun20_infinite}
D.~Kundu.
\newblock Growth of fine {Selmer} groups in infinite towers.
\newblock {\em Canad. Math. Bull.}, 63(4):921--936, 2020.

\bibitem{Kun20_p}
D.~Kundu.
\newblock Growth of $p$-fine {Selmer} groups and $p$-fine {Shafarevich-Tate}
  groups in {$\mathbb{Z}/p\mathbb{Z}$} extensions.
\newblock {\em J. Ramanujan Math. Soc}, 2020.

\bibitem{Kun20_uniform}
D.~Kundu.
\newblock Growth of {Selmer} groups and fine {Selmer} groups in uniform pro-$p$
  extensions.
\newblock {\em Annales math{\'e}matiques du Qu{\'e}bec}, pages 1--16, 2020.

\bibitem{LM15}
M.~F. Lim and V.~K. Murty.
\newblock The growth of fine {Selmer} groups.
\newblock {\em J. Ramanujan Math. Society}, 31(1):79--94, 2016.

\bibitem{Mat18}
A.~Matar.
\newblock Fine {S}elmer groups, {H}eegner points and anticyclotomic
  $\mathbb{Z}_{p}$-extensions.
\newblock {\em Int. J. Number Theory}, 14(05):1279--1304, 2018.

\bibitem{Maz72}
B.~Mazur.
\newblock Rational points of abelian varieties with values in towers of number
  fields.
\newblock {\em Invent. Math.}, 18(3-4):183--266, 1972.

\bibitem{NSW08}
J.~Neukirch, A.~Schmidt, and K.~Wingberg.
\newblock {\em Cohomology of number fields}, volume 323.
\newblock Springer, 2013.

\bibitem{Och09}
Y.~Ochi.
\newblock A remark on the pseudo-nullity conjecture for fine {S}elmer groups of
  elliptic curves.
\newblock {\em Rikkyo Daigaku sugaku zasshi}, 58(1):1--7, 2009.

\bibitem{PR93}
B.~Perrin-Riou.
\newblock Fonctions ${L}$ $p$-adiques d’une courbe elliptique et points
  rationnels.
\newblock {\em Annales de l'institut Fourier}, 43(4):945--995, 1993.

\bibitem{PR95}
B.~Perrin-Riou.
\newblock Fonctions ${L}$ p-adiques des repr{\'e}sentations $p$-adiques.
\newblock {\em Ast{\'e}risque}, (229), 1995.

\bibitem{PW11}
R.~Pollack and T.~Weston.
\newblock On anticyclotomic $\mu$-invariants of modular forms.
\newblock {\em Comp. Math.}, 147(5):1353--1381, 2011.

\bibitem{Rub91}
K.~Rubin.
\newblock The ``main conjectures” of {Iwasawa} theory for imaginary quadratic
  fields.
\newblock {\em Invent. Math.}, 103(1):25--68, 1991.

\bibitem{Rub14}
K.~Rubin.
\newblock {\em Euler Systems.(AM-147)}, volume 147.
\newblock Princeton University Press, 2014.

\bibitem{Vat03}
V.~Vatsal.
\newblock Special values of anticyclotomic ${L}$-functions.
\newblock {\em Duke Math.}, 116(2):219--261, 2003.

\bibitem{Wut07}
C.~Wuthrich.
\newblock The fine {Tate}--{Shafarevich} group.
\newblock In {\em Math. Proc. Camb. Philos. Soc.}, volume 142, pages 1--12.
  Cambridge University Press, 2007.

\bibitem{Wut05}
C.~Wuthrich.
\newblock Iwasawa theory of the fine {Selmer} group.
\newblock {\em J. Alg. Geom.}, 16(1):83--108, 2007.

\end{thebibliography}
\end{document}